\documentclass[11pt]{amsart}
\usepackage{fullpage}
\usepackage{amsmath,amssymb,amsthm,hyperref}
\usepackage{color}

\theoremstyle{plain}
\newtheorem{thm}{Theorem}[section]
\newtheorem{lem}[thm]{Lemma}
\newtheorem{prop}[thm]{Proposition}

\theoremstyle{remark}

\numberwithin{equation}{section}

\newcommand{\R}{\mathbb{R}}
\newcommand{\C}{\mathbb{C}}
\newcommand{\N}{\mathbb{N}}

\newcommand{\Z}{\mathbb{Z}}
\newcommand{\ga}{\gamma}
\newcommand{\ep}{\varepsilon}

\DeclareMathOperator*{\dist}{\mathrm{dist}}

\begin{document}    
	
	\title[From counting to decoupling]{Reversing a philosophy: \\from counting   to square functions and decoupling}
	\author[P. T. Gressman \ \ S. Guo \  \ L. B. Pierce \ \ J. Roos \ \ P.-L. Yung] {Philip T. Gressman \ \ Shaoming Guo \ \ Lillian B. Pierce\\ \ \ Joris Roos \ \ Po-Lam Yung}
	\date{\today}
	
	\address{Philip T. Gressman: Department of Mathematics, University of Pennsylvania, Philadelphia, PA-19104, USA}
	\email{gressman@math.upenn.edu}
	
	\address{Shaoming Guo: Department of Mathematics, University of Wisconsin-Madison, Madison, WI-53706, USA}
	\email{shaomingguo@math.wisc.edu}
	
	\address{Lillian B. Pierce: Department of Mathematics, Duke University, Durham, NC-27708, USA}
	\email{pierce@math.duke.edu}
	
	\address{Joris Roos: Department of Mathematics, University of Wisconsin-Madison, Madison, WI-53706, USA}
	\email{jroos@math.wisc.edu}
	
	\address{Po-Lam Yung: Department of Mathematics, The Chinese University of Hong Kong, Ma Liu Shui, Shatin, Hong Kong}
	\email{plyung@math.cuhk.edu.hk}
	
	
	\begin{abstract}
		Breakthrough work of Bourgain, Demeter, and Guth recently established that decoupling inequalities can prove powerful results on counting integral solutions to systems of Diophantine equations. In this note we demonstrate that in appropriate situations this implication can also be reversed. As a first   example, we observe that a count for the number of integral solutions to a  system  of Diophantine equations implies a discrete decoupling inequality. 
		Second, in our main result we prove an $L^{2n}$ square function estimate (which implies a corresponding decoupling estimate) for the extension operator associated to a non-degenerate curve in $\R^n$. The proof is via a combinatorial argument that builds on the idea   that  if $\gamma$ is a non-degenerate curve in $\R^n$, then as long as $x_1,\ldots, x_{2n}$ are chosen from a sufficiently well-separated set, then
		$ \gamma(x_1)+\cdots+\gamma(x_n) = \gamma(x_{n+1}) + \cdots + \gamma(x_{2n}) $
		essentially only admits solutions in which $x_1,\ldots,x_n$ is a permutation of $x_{n+1},\ldots, x_{2n}$.
	\end{abstract}

	\dedicatory{Dedicated to Elias M. Stein,\\ in deep appreciation of his generous teaching and clear-sighted vision in harmonic analysis.}
	
	\maketitle
	
	\section{Introduction}
	
	In celebrated work, Bourgain, Demeter, and Guth \cite{BDG16} established  a sharp decoupling inequality for the moment curve, and thereby deduced a full proof of the Vinogradov Mean Value Theorem, providing a count for the number of integral solutions $ 1 \leq x_1,\ldots, x_{2s} \leq X$ to the Vinogradov system 
	\begin{align} 
	x_1 + \cdots + x_s &= x_{s+1} + \cdots + x_{2s} \nonumber \\
	x_1^2 + \cdots + x_s^2 &= x_{s+1}^2 + \cdots + x_{2s}^2 \nonumber \\
	&\vdots \nonumber \\
	x_1^n + \cdots + x_s^n &= x_{s+1}^n + \cdots + x_{2s}^n. \label{eqn:vinogradovsystem}
	\end{align}
	In this note, we show that in appropriate regimes, this implication can be reversed, with a count for the number of integral solutions to (\ref{eqn:vinogradovsystem}) implying a corresponding decoupling inequality. 
	
First, we prove a simple example of this philosophy: we deduce a discrete decoupling estimate from an assumed count for solutions to a system of Diophantine equations such as (\ref{eqn:vinogradovsystem}); this follows from a restricted weak-type estimate and comparisons of discrete norms.

Our main result is in a more general setting: in place of the moment curve $(t,t^2, \ldots, t^n)$, which leads to the system (\ref{eqn:vinogradovsystem}), we   consider any non-degenerate 
	$C^n$ curve $\gamma:[0,1]\to \R^n$ with $n\ge 2$. We prove that an extension operator associated to $\gamma$ satisfies a square function estimate (or reverse Littlewood-Paley inequality) for $L^{2n}$, which immediately implies an $\ell^2$ decoupling inequality in $L^{2n}$. 
	Our proof is combinatorial in nature, and capitalizes on an observation that since $\gamma$ is non-degenerate, as long as $x_1,\ldots, x_{2n}$ are chosen from a sufficiently well-separated set, then 
	$ \gamma(x_1)+\cdots+\gamma(x_n) = \gamma(x_{n+1}) + \cdots + \gamma(x_{2n}) $
	essentially only admits solutions in which $x_1,\ldots,x_n$ is a permutation of $x_{n+1},\ldots, x_{2n}$.
We now state these results precisely.

\subsection{Counting implies discrete decoupling}
	
	Given a map $\phi:\N\to \Z^n$ and an integer $s\ge 1$ let us consider the system of $n$ equations given by
	\begin{equation}\label{eqn:generalsystem}
	\phi(x_1) + \cdots + \phi(x_s) = \phi(x_{s+1}) + \cdots + \phi(x_{2s}).
	\end{equation}
	For every finite set $\mathcal{S}$ of positive integers let $J_{s,\phi}(\mathcal{S})$ denote the number of solutions $(x_1,\dots,x_{2s})\in \mathcal{S}^{2s}$ of the system \eqref{eqn:generalsystem}. Fix $N$ and consider an arbitrary subset $\mathcal{S} \subseteq \{1,\ldots, N\}$.  We see an immediate lower bound $J_{s,\phi}(\mathcal{S}) \ge s! |\mathcal{S}|^s + O(|\mathcal{S}|^{s-1}),$ since solutions for which $x_1,\ldots, x_s$ is a permutation of $x_{s+1},\ldots, x_{2s}$ always exist trivially (the diagonal solutions). A trivial upper bound is $J_{s,\phi}(\mathcal{S}) \leq |\mathcal{S}|^{2s}$.
	One route towards obtaining better upper bounds for the quantity $J_{s,\phi}(\mathcal{S})$ is via a discrete $\ell^p$ decoupling inequality for $L^{2s}$, which is a statement of the following form: given $s \geq 1$, $p \geq 1$, there exists a constant $C_{s,p,\phi,N}$ such that for all sequences $a= (a_j)_j \in \C^N$, 
	\begin{equation}\label{eqn:discdec}
	\big\| \sum_{j=1}^N a_j e(\phi(j)\cdot \alpha) \big\|_{L^{2s}([0,1]^n)} \le C_{s,p,\phi,N} \big(\sum_{j=1}^N |a_j|^p\big)^{1/p}.
	\end{equation}
(To see precisely that this takes the standard form of a decoupling inequality, notice that on the right-hand side, $|a_j| = \|a_j e(\phi(j)\cdot \alpha) \big\|_{L^{2s}([0,1]^n)}.$)
	For any subset $\mathcal{S}$, upon setting $a = (a_j)_j=\mathbf{1}_{\mathcal{S}}$, the inequality \eqref{eqn:discdec} implies the bound 
	\[ J_{s,\phi}(\mathcal{S}) \le C_{s,p,\phi,N}^{2s} | \mathcal{S}|^{2s/p}. \]
	As our first point, we make the simple observation that a converse also holds.
	
	\begin{thm}\label{thm_equiv}
	Given a map $\phi: \N \to \Z^n$ and an integer $s \geq 1$, suppose that there exists a constant $\theta = \theta(\phi,s) \in [s,2s)$ and a constant $c= c(\phi,s) \in (0,\infty)$ such that for all  $N\ge 1$ and for all subsets $\mathcal{S}\subset \{1,\dots,N\}$ we have the inequality 
		\begin{equation}\label{eqn:countingasm}
		J_{s,\phi}(\mathcal{S})\le c |\mathcal{S}|^{\theta}.
		\end{equation}
		Then the $\ell^p$ decoupling inequality for $L^{2s}$ holds for $p=\frac{2s}{\theta}\in (1,2]$: namely,  there exists a constant $c'$ such that for every $(a_j)_j\in \C^N$, we have 
		\begin{equation}\label{eqn:discdecclaim}
		\big\| \sum_{j=1}^N a_j e(\phi(j)\cdot \alpha) \big\|_{L^{2s}([0,1]^n)} \le c' (1+p^{-1}(\log\,N)^{\frac1{p'}})\big(\sum_{j=1}^N |a_j|^p\big)^{1/p}.
		\end{equation}
Here we have $1/p + 1/p'=1$, and we may take $c'=2^{1/p}4^{1/p'}c^{1/2s}$.  
	\end{thm}
If it is known for a certain function $\phi$ that in the above setting we may take $\theta=s$ (that is, all solutions are diagonal solutions), this statement is a discrete analogue of our main result, which we now describe.

\subsection{Counting implies a square function estimate}

We now define the notation required to state our main result.
	Recall that a $C^n$ curve $\gamma:[0,1]\to \R^n$   is said to be non-degenerate if 
	\begin{equation}\label{eqn:nondeg}
	\det(\gamma'(t),\gamma''(t),\dots,\gamma^{(n)}(t))\not=0\;\text{for every}\;t\in [0,1].
	\end{equation}
	A typical example is the moment curve
	\[ \gamma(t) = (t,t^2,\dots,t^n). \]
	Given any such curve, we may define the associated Fourier extension operator
	\[ E_I f(x) = \int_I e^{2\pi i x\cdot \gamma(t)} f(t) dt \qquad(x\in\R^n),\]
	where $I\subset [0,1]$ is an interval.
	Given a ball $B\subset\R^n$ of radius $R$ centered at a point $x_0\in\R^n$ we define a weight localized near $B$ by
	\[ w_B(x) = (1+R^{-1}|x-x_0|)^{-E}, \]
	where $E>n$ is fixed once and for all ($E=n+1$ suffices). Given any non-negative function $v$ we define the weighted $L^p$ norm
	\[ \|f\|_{L^p(v)} = \big(\int_{\R^n} |f(x)|^p v(x) dx\big)^{1/p}. \]
	
	Our main result is the following square function estimate.
	\begin{thm}\label{thm:main}
		Suppose that $\gamma:[0,1] \rightarrow \mathbb{R}^n$ is a non-degenerate $C^n$ curve. Then there exists a constant $C = C(\gamma, n)\in (0,\infty)$ such that the following holds: for each  integer $1 \leq m \leq n$, for every $R\ge 1$ and every ball $B$ of radius at least $R^n$, we have that for all $f \in L^{2m}(w_B)$, 
		\begin{equation}\label{eqn:main}
		\|E_{[0,1]} f\|_{L^{2m}(w_B)} \le C\big\|\big(\sum_{|I|=R^{-1}} |E_I f|^2\big)^{1/2}\big\|_{L^{2m}(w_B)},
		\end{equation}
	where the summation is over intervals $I$ belonging to a dissection of $[0,1]$ into intervals of length $R^{-1}$. 
	\end{thm}
	
	In the case $n=m=2$, the estimate \eqref{eqn:main} is classical and in its essence goes back to inequality (4) in Fefferman \cite{Fef73}. 
	
	For comparison, we recall the shape of an $\ell^2$ decoupling inequality for $L^{p}$, which is the statement that for  
	every ball $B$ of radius at least $R^n$ and every $\varepsilon>0$ there exists a constant $C_\ep$ such that for all $f \in L^p(w_B)$,
	\begin{equation}\label{eqn:decoupling}
	\|E_{[0,1]} f\|_{L^p(w_B)} \le C_\varepsilon R^{\varepsilon} \big(\sum_{|I|=R^{-1}} \|E_I f\|_{L^{p}(w_B)}^2\big)^{1/2}.
	\end{equation}
	Minkowski's inequality shows that for $p \geq 2$, a square function estimate in $L^p$ implies the corresponding $\ell^2$ decoupling for $L^p$ (and is strictly stronger if $p>2$, see \cite[\S 5.3.2]{Pie19} for an explanation), so that (\ref{eqn:main}) immediately implies (\ref{eqn:decoupling}) in the case that $p$ is an even integer with $2 \leq p\le 2n$.
	
	Of course, the deep work of Bourgain--Demeter--Guth \cite{BDG16} proved the result (\ref{eqn:decoupling}) of $\ell^2$ decoupling for $L^p$  in the much larger, sharp, range $2 \leq p \leq n(n+1)$, which then implies the truth of the main conjecture in the context of Vinogradov's mean value theorem. (See also the work of Wooley, which   resolves this major conjecture by other methods \cite{Woo16,Woo17}.)
	
	Yet relative to  this broader context, Theorem \ref{thm:main} has two appealing aspects: first,  in the case $2<p\le 2n$ it is a strengthening of the decoupling inequality, and moreover  our argument is surprisingly simple, critically using the fact that $p$ is an even integer.
	
	We can already see a hint of the special role of the exponent $p=2n$ from the following.
	Fix an integer $s\ge 1$. For every integer $X \ge 1$ we let $J_{s,n}(X)$ denote the number of integral solutions $(x_1,\dots,x_{2s})$ with $1\le x_j\le X$ to the system of equations given by (\ref{eqn:vinogradovsystem}).
	Certainly, any tuples in which $x_1,\ldots, x_s$ is a permutation of $x_{s+1},\ldots, x_{2s}$ provide a solution, and these are referred to as diagonal solutions, of which there are $s! X^s + O(X^{s-1})$ in number. Moreover, if $s \leq n$ (corresponding to looking at $L^{2s}$ spaces with even $2s \leq 2n$), then it has long been known that these are  the \emph{only solutions} to (\ref{eqn:vinogradovsystem}); as this idea is a central motivation for our work, we review a proof of this classical fact in Lemma \ref{lemma_diag}. 
	More generally for non-degenerate $C^n$ curves $\gamma: [0,1] \rightarrow \R^n$, our approach to proving Theorem \ref{thm:main} for even integers $p \leq 2n$ uses a perturbed version of this fact,  that the system of $n$ equations given by
	\begin{equation}\label{iterated}
	\gamma(x_1)+\cdots+\gamma(x_n) = \gamma(x_{n+1}) + \cdots + \gamma(x_{2n}) 
	\end{equation}
	only admits ``essentially diagonal'' solutions.
	
To motivate precisely the result we prove in this context, we first remark on the use 
	 of the localized norms, weighted by the functions $w_B$, employed in Theorem \ref{thm:main}. The extension operator $E_I$ has been studied extensively in the literature, often in the guise of its dual, the restriction operator $f\mapsto \widehat{f}|_{\gamma(I)}$ (see e.g.  \cite{BOS09} for a survey of related literature).  Drury \cite[Thm. 2]{Dru85} proved that for a non-degenerate curve $\ga: I \rightarrow \R^n$ on an interval $I \subset [0,1]$,
	\[ \|E_{I} f\|_{L^p(\R^n)} \le c_p \| f\|_{L^q(I) } \leq c_p \|f\|_{L^\infty(I)} \]
	holds for all $p>\frac{n(n+1)}2+1$ (here $q$ is defined by its conjugate $q'$ satisfying $q'n(n+1)/2 = p$). This result is sharp in the range of $p$, since it is known for example in the case of $\ga$ being the moment curve, that $\|E_I 1\|_{L^p(\R^n)}=\infty$ if $p\le \frac{n(n+1)}2+1$ (recorded in  \cite[Thm. 1.3]{ACK87}, arising from earlier work \cite{ACK79}). 
	This shows in particular that unless we localize using the weight $w_B$, the main inequality (\ref{eqn:main}) would lose its significance, since both sides would be infinite. 
	
In general, a weighted norm such as 
\[ \| E_I f \|_{L^{2m}(\phi)}^{2m} = \int  |E_I f(x)\phi(x)^{1/(2m)}|^{2m} dx\]
	leads us to study, on the Fourier side, the convolution of $(E_If)\hat{\;}$ with $(\phi(x)^{1/(2m)})\hat{\;}$, which has the effect of ``blurring'' the support of of $(E_If)\hat{\;}$, so that we must consider not only exact solutions to (\ref{iterated}) but also near-solutions to (\ref{iterated}).
	(See equation (\ref{blur_demo}) for the precise line in our argument at which this occurs, or see \cite[\S 8.1.3]{Pie19} for another example of this effect.) 
This leads us to prove the following key result, which shows that any near-solution to (\ref{iterated}) must be essentially diagonal.
	\begin{prop}\label{prop_ga_diag}
		Let $\ga : [0,1]\rightarrow \R^n$ be a non-degenerate $C^n$ curve. Then there exist constants $\delta_0 = \delta_0(\ga,n) \leq 1$ and $c_0= c_0(\ga,n) \geq 10$ such that the following holds. Let  $\mathcal{I}$ be any set of intervals   from a dissection of $[0,1]$ into pairwise disjoint intervals of length $R^{-1}$, such that
		\begin{equation}\label{assps}
		 \mathrm{dist}(I,I')\ge c_0 R^{-1} \quad\text{for}\;I\not=I'\in\mathcal{I}, \quad\text{and}\quad \mathrm{diam}\big(\bigcup_{I\in\mathcal{I}} I\big)\le \delta_0. 
		 \end{equation}
		Then for any collection of $2n$ intervals $I_1,\ldots, I_n, I_1', \ldots, I_n'$ from $\mathcal{I}$, if the tuple $(I_1,\dots,I_n)$ is not a permutation of $(I_1',\dots,I_n')$, then for any points $t_i\in I_i$ and $s_i\in I_i'$,
		\begin{equation}\label{eqn:gammaest}
		\big|\sum_{i=1}^n (\gamma(t_i) - \gamma(s_i)) \big| \ge R^{-n}.
		\end{equation}
	\end{prop}
This proposition comprises the majority of the technical work of the paper; once it has been proved, the square function estimate of Theorem \ref{thm:main} quickly follows.
	
	It is reasonable to ask whether a square function estimate of the form (\ref{eqn:main}) can be proved for $p> 2n$, leading to consideration of the system (\ref{eqn:vinogradovsystem}) (or its generalization (\ref{iterated}) for non-degenerate $\ga$) in $2s$ variables, for $s>n$.
	Our present method of argument seems to rely on being in a regime in which the only solutions to (\ref{eqn:vinogradovsystem}) (or near-solutions to (\ref{iterated})) are diagonal (or essentially diagonal), and so this leads one to ask whether there are off-diagonal solutions to (\ref{eqn:vinogradovsystem}) when $s>n$, and if so, how many. This has been studied since the 1850's; as we later note,
	 for $1 \leq n \leq 9$ and $n=11$, it is known that at least one off-diagonal solution already exists to (\ref{eqn:vinogradovsystem}) when $s=n+1$ (see \S \ref{sec_elem}).  Moreover, it is known (Lemma \ref{lemma_off_diag}) that as soon as one off-diagonal solution to (\ref{eqn:vinogradovsystem}) exists, it generates many more, thus presenting a significant obstacle to our current method of proof.  
	This seems to suggest that the particular exponent $2n$, despite being far away from  the sharp decoupling exponent $p_n = n(n+1)$, still plays a special role in square function estimates such as Theorem \ref{thm:main}.

	\subsection*{Notation} We will use the notation $A\lesssim B$ to denote that $A\le C\cdot B$ for some constant $C$. The constant $C$ may change from line to line and is allowed to depend on $\gamma$ and $n$. Given two intervals $J_1,J_2$ on the real line, we will say that they are essentially disjoint if they are disjoint except possibly at their endpoints. 
	We define the distance between two intervals by $\dist(J_1,J_2)=\inf_{x\in J_1, y\in J_2} |x-y|$ and the diameter of an interval by $\mathrm{diam}(J)=\sup_{x,y\in J}|x-y|$; we denote the center of an interval $J$ by $c(J)$.

	\subsection*{Acknowledgements.} We thank the American Institute of Mathematics for funding our collaboration in the context of a  SQuaRE workshop series. Gressman has been partially supported by NSF Grant DMS-1764143. Pierce has been partially supported by NSF CAREER grant DMS-1652173, a Sloan Research Fellowship, and as a von Neumann Fellow at the Institute for Advanced Study, by the Charles Simonyi Endowment and NSF Grant No. 1128155. Yung was partially supported by a General Research Fund CUHK14303817 from the Hong Kong Research Grant Council, and a direct grant for research from the Chinese University of Hong Kong (4053341).	
	
\section{Elementary arguments for the moment curve}\label{sec_elem}	
This section presents classical arguments about diagonal and off-diagonal solutions for the Vinogradov system (\ref{eqn:vinogradovsystem}) relating to the moment curve $\ga(t) = (t,t^2,\ldots, t^n)$.
 
	\begin{lem}\label{lemma_diag}
		For $s\leq n$, the only solutions $1 \leq x_1 , \ldots, x_{2s} \leq X$ to (\ref{eqn:vinogradovsystem}) are diagonal, that is, $x_1,\ldots, x_s$ is a permutation of $x_{s+1},\ldots, x_{2s}$. 
	\end{lem}
	\begin{proof}
	The proof relates back to identities known to Newton; we follow the presentation of  \cite[\S 21.9]{HW08}.	First note that it suffices to consider the case $s=n$, since (upon setting the remaining variables to zero) any violation of the result for fewer variables would result in a violation of the result for $s=n$. 
		
		For each $1 \leq j \leq n$ let us denote by $p_j(t_1,\ldots, t_n)$ the polynomial $\sum_{1 \leq i \leq n} t_i^j $. For each $ 1\leq j \leq n$ let us denote by $S_j(t_1,\ldots, t_n)$ the $j$-th elementary symmetric polynomial, so that $S_0(t_1,\ldots, t_n) = 1,$ $S_1(t_1,\ldots, t_n) = \sum_i t_i,$ $S_2(t_1,\ldots, t_n) = \sum_{i<i'} t_it_{i'},$and so on up to $S_n(t_1,\ldots, t_n) = t_1\cdots t_n$. In particular, given any values for $(t_1,\ldots,t_n)$, the symmetric polynomials have the property that the monic one-variable polynomial $F_{t_1,\ldots, t_n}(T)$  with roots $t_1,\ldots, t_n$ is given by 
		\[ S_0(t_1,\ldots, t_n) T^n + \cdots + S_{n-1}(t_1,\ldots, t_n) T + S_n (t_1,\ldots, t_n).\]
		The Newton-Girard identities state that for each $j$, $S_j$ may be determined from the polynomials $S_i$ with $i < j$ and $p_i$ with $i  \leq j$; precisely, we have the statement that for each $1 \leq j \leq n$, 
		\begin{equation}\label{SP} 
		j S_j(t_1,\ldots, t_n) = \sum_{i=1}^j (-1)^{i-1} S_{j-i}(t_1,\ldots, t_n)p_i(t_1,\ldots, t_n).
		\end{equation}
		Now on the one hand, if we assume that $(x_1 , \ldots, x_{2n})$ solves (\ref{eqn:vinogradovsystem}), then we know that for each $1 \leq j \leq n$ we have 
		\[ p_j(x_1,\ldots, x_n) = p_j(x_{n+1}, \ldots x_{2n}).\]
		But by (\ref{SP}), we therefore see that for each $1 \leq j \leq n$, 
		\[ S_j(x_1,\ldots, x_n) = S_j(x_{n+1}, \ldots x_{2n}).\]
		Thus, recalling our earlier notation,   $F_{x_1,\ldots, x_n}(T)$ and $F_{x_{n+1},\ldots, x_{2n}}(T)$ are identical as polynomials in $T$ and hence the roots $(x_1,\ldots, x_n)$ of the first polynomial are a permutation of the roots $(x_{n+1},\ldots, x_{2n})$ of the second polynomial, proving the lemma.
	\end{proof}

Next we see that if there is even one off-diagonal solution to (\ref{eqn:vinogradovsystem}), it can be used to generate many more. Here we recall a proof in \cite[p. 194]{VW97}; we thank Trevor Wooley for pointing out that similar ideas may be found in Mordell \cite{Mor32} and Gloden \cite{Glo44}.

	\begin{lem}\label{lemma_off_diag}
		Suppose that an off-diagonal solution $(x_1,\ldots, x_{2s})$ to (\ref{eqn:vinogradovsystem}) exists, with $1 \leq x_1 , \ldots, x_{2s} \leq X$. Then there are at least $\gtrsim X^2$ off-diagonal solutions in this range.
	\end{lem}
	\begin{proof}
	The system (\ref{eqn:vinogradovsystem}) is translation-dilation invariant, so that a particular tuple $\mathbf{x}$ is a solution if and only if $q \mathbf{x} + \mathbf{h}$ is, for any dilation factor $q$ and any shift $\mathbf{h}$ (see e.g. \cite[\S 3.5]{Pie19}). Let $\mathbf{x}= (x_1,\ldots,x_{2s})$ be the presumed off-diagonal solution. Then set $\mathbf{h} = (h,h,\ldots, h)$; for any $1 \leq q < X/\max\{x_i\}$ and any $1 \leq h \leq X-q\max\{x_i\}$ we have that $ \mathbf{y} = q\mathbf{x} + \mathbf{h}$ is also an off-diagonal  solution with each entry $1 \leq y_i \leq X$. Given a particular off-diagonal solution $\mathbf{x}$, this yields $\gtrsim_{\max\{x_i\}} X^2$ distinct off-diagonal solutions. The dependence in this lower bound on $\max\{x_i\}$ is allowable, since $\mathbf{x}$ is fixed once and for all, and we may take $X$ to be arbitrarily large.
	\end{proof}

This phenomenon, combined with the importance to our proof that we are in the purely-diagonal regime, leads us to ask: given $n$, what is the least $s$ for which there is at least one off-diagonal solution to (\ref{eqn:vinogradovsystem})? This is a case of the classical Prouhet-Tarry-Escott problem, which  remains open in general, since early work in the  1850's (see e.g. \cite[\S 21.9]{HW08}). If we denote by $P(n)$ the least such $s$, then Lemma \ref{lemma_diag} shows that $P(n) \geq n+1$. It is known for all $n$ that $P(n) \leq n(n+1)/2 +1$ \cite[Thm. 409]{HW08}, but one might expect that it can be significantly smaller. In fact for $1 \leq n \leq 9$ and $n=11$, specific off-diagonal solutions have been exhibited by various authors, which confirm that $P(n) = n+1$ in these cases; see the  end-notes to the discussion in \cite[\S 21.9]{HW08}.
	Thus, by Lemma \ref{lemma_off_diag}, a method of proof that aims to obtain a square function estimate analogous to Theorem \ref{thm:main} for $L^{2n+2}$ must be able to accommodate a significant presence of off-diagonal solutions.
	
\section{Proof of Theorem \ref{thm_equiv}: equivalence between discrete decoupling and counting}\label{sec_discrete}

	For the proof of Theorem \ref{thm_equiv} we begin by observing that since 
	\[J_{s,\phi}(\mathcal{S}) = \big\| \sum_{j\in \mathcal{S}} e(\phi(j)\cdot \alpha)\big\|_{L^{2s}([0,1]^n)}^{2s},\]
	 the assumption \eqref{eqn:countingasm} is equivalent to the statement that the inequality
	\begin{equation}\label{a_ineq}
	 \big\| \sum_{j=1}^N a_j e(\phi(j)\cdot \alpha) \big\|_{L^{2s}([0,1]^n)} \le c^{1/2s} \big(\sum_{j=1}^N |a_j|^p\big)^{1/p}
	 \end{equation}
	holds for all $a=(a_j)_j$ of the form $a = \mathbf{1}_{\mathcal{S}}$ for some $\mathcal{S}\subset \{1,\dots,N\}$, where $p=2s/\theta$. 
		We recall the definition of the norms 
	\[ \|a\|_{\ell^p}=(\sum_{j=1}^N |a_j|^p)^{1/p} = p^{1/p} (\int_0^\infty s^{p-1} \lambda_a(s)ds)^{1/p}\]
	and
 \[\|a\|_{\ell^{p,1}}=\int_0^\infty \lambda_a^{1/p}(s) ds,\]
  where $\lambda_a(s) = \# \{ j \in \{1,\ldots,N\} \ : \ |a_j| > s \}$.
	Upon defining a function $T : \C^N \to [0,\infty)$ 
	by 
	\[T(a) = \big\| \sum_{j=1}^N a_j e(\phi(j)\cdot \alpha) \big\|_{L^{2s}([0,1]^n)},\]
 the inequality (\ref{a_ineq}) can be written as the statement that $T(\mathbf{1}_\mathcal{S})\leq c^{1/2s}\| \mathbf{1}_S\|_{\ell^p}$ holds for all $\mathcal{S}\subset \{1,\dots,N\}$. We then obtain Theorem \ref{thm_equiv} by an application of the following general fact.
	
	\begin{lem}\label{lem:weaknorm}
		Let $p\in (1,\infty), c\in (0,\infty)$ and let $T:\C^N\to [0,\infty)$ be a sublinear function such that 
		\begin{equation}\label{T_prop}
		T(\mathbf{1}_\mathcal{S})\leq C \|\mathbf{1}_\mathcal{S}\|_{\ell^p} \; \text{ holds for all $\mathcal{S}\subset \{1,\dots,N\}$.}
		\end{equation}
		 Then
		\[ T(a) \le c' (1+(\log\,N)^{1/p'}/p) \|a\|_{\ell^p} \]
		holds for all $a\in \C^N$, where $c' = 2^{1/p}4^{1/p'}C$.
	\end{lem}
The first step to prove Lemma \ref{lem:weaknorm}  is  the observation that, as in the general Lorentz space theory (see e.g. \cite[Ch. V, \S 3]{SW71}), the restricted weak-type hypothesis (\ref{T_prop}) implies the estimate $T(a)\leq C \|a\|_{\ell^{p,1}}$ for any $a \in \C^N$ with non-negative entries. 
Thus given a general   $a \in \C^N$, we split it into real and imaginary parts $a_r,a_i$ and then, respectively, positive and negative parts, say $a_r^+,a_r^-,a_i^+,a_i^-$; then using the assumed sublinearity of $T$, we see that 
\[ T(a) \leq C\{  \|a_r^+\|_{\ell^{p,1}} +  \|a_r^-\|_{\ell^{p,1}} + \|a_i^+\|_{\ell^{p,1}} + \|a_i^-\|_{\ell^{p,1}} \}  . \]
What remains is to dominate each weak-type $\ell^p$ norm by the corresponding $\ell^p$ norm, which follows from applying Lemma \ref{lemma_norm} (below) term by term, followed by H\"older's inequality to the sum of four terms, resulting in
	\[ T(a) \leq 4^{1/p'} C ( 1 + p^{-1}(\log N)^{1/p'}))\{  \|a_r^+\|_{\ell^{p}}^p +  \|a_r^-\|_{\ell^{p}}^p + \|a_i^+\|_{\ell^{p}}^p + \|a_i^-\|_{\ell^{p}}^p \}^{1/p} .\]
Using the disjoint supports of $a_r^+$ and $a_r^-$, and similarly for $a_i^+$ and $a_i^-$, the right-hand side is equal to
\[4^{1/p'} C ( 1 + p^{-1}(\log N)^{1/p'}))\{  \|a_r^+  - a_r^{-}\|_{\ell^{p}}^p  +\|a_i^+ - a_i^-\|_{\ell^{p}}^p \}^{1/p} ,\]
which is in turn bounded above by 
\[2^{1/p}4^{1/p'} C ( 1 + p^{-1}(\log N)^{1/p'}))\|a\|_{\ell^p} ,\]
completing the proof of Lemma \ref{lem:weaknorm}.
	\begin{lem}\label{lemma_norm}
		For any $ a\in \C^N$ and any $p \in [1,\infty)$ we have
		\[ \|a\|_{\ell^{p,1}} \leq \big( 1 + p^{-1}(\log N)^{1/p'} \big)  \|a\|_{\ell^{p}}. \]
	\end{lem}
	
	\begin{proof}
		By Chebyshev's inequality, $\lambda_a^{1/p}(s)\le s^{-1} \|a\|_{\ell^p}$ for all $s>0$.
		Observe that $\lambda_a(s)$ is a non-negative integer no greater than $N$; in particular, it must be zero if it is less than one, which implies that $\lambda_a(s) = 0$ for $s > \|a\|_{\ell^{p}}$. Therefore,
		\[
		\begin{aligned}
		\int_0^\infty \lambda_a^{1/p} (s) ds & \leq \int_0^{N^{-\frac{1}{p}} \|a\|_{\ell^{p}}} N^{1/p} ds  + \int_{N^{-\frac{1}{p}} \|a\|_{\ell^{p}}}^{\|a\|_{\ell^{p}}} \lambda_a ^{1/p}(s) ds\\
		& \le \|a\|_{\ell^p} + \big( \int_0^\infty s^{p-1} \lambda_a(s) ds \big)^{1/p} \big( \int_{N^{-\frac1p}\|a\|_{\ell^p}}^{\|a\|_{\ell^p}}s^{-1} ds \big)^{1/p'}\\
		& = \big( 1 + p^{-1}(\log N)^{1/p'} \big) \|a\|_{\ell^p},
		\end{aligned}\]
		where we have applied H\"older's inequality in the penultimate step.
	\end{proof}

	\section{Non-degenerate curves: linear independence of derivatives at separated points  }
	
	In this section we begin the proof of Proposition \ref{prop_ga_diag} by proving two results on the linear independence of derivatives of $\ga'(t)$ when $t$ is evaluated at distinct points. 
	The first result is motivated by an observation 
	in the special case $\gamma(t)=(t,t^2/2,\cdots,t^n/n)$: the Vandermonde determinant shows that
	for any $u_1,\dots,u_n\in\R$,
	\begin{equation}\label{Vandermonde}
	\det(\gamma'(u_1),\cdots,\gamma'(u_n)) = \prod_{1\le i<j\le n} (u_j - u_i) .
	\end{equation}
	Thus in particular if the points $u_j$ are separated, the determinant is well-controlled. We now prove comparable   upper and lower bounds for this determinant, in the general case of a non-degenerate curve $\ga$. 
	
	\begin{prop}\label{lem:detcond}
		Let $\gamma: [0,1] \to \R^n$ be a $C^n$ curve.\\
		(a) There exists a constant $C = C(\ga,n)$ such that for every $0<u_1<\dots<u_n<1$ we have
		\begin{equation}\label{eqn:detcond0}
		|\det(\gamma'(u_1), \dots, \gamma'(u_n))| \leq C \prod_{1\le i<j\le n} (u_j-u_i).
		\end{equation}
		(b) Suppose furthermore that $\gamma$ is non-degenerate. Then there exists a constant $C' = C'(\ga,n)$ and  $\delta_0=\delta_0(\gamma,n)>0$ such that for every $0<u_1<\dots<u_n<1$ with $u_n-u_1<\delta_0$ we have
		\begin{equation}\label{eqn:detcond1}
		|\det(\gamma'(u_1), \dots, \gamma'(u_n))| \geq C' \prod_{1\le i<j\le n} (u_j-u_i).
		\end{equation}
	\end{prop}
	Substantially more refined estimates of the type exhibited in Proposition \ref{lem:detcond} have been obtained recently in  \cite{DLW09,DW10} in the case of polynomial curves.
	However, we do not require such a refined estimate, and we give in this section a direct proof of the proposition, which does not require a delicate decomposition of $\R$.
	
	Furthermore, we prove a version of Proposition \ref{lem:detcond} that is averaged over certain intervals. 
	We will use the convention that an expression such as $\int_{J} \gamma'(u)  du$ denotes a column vector, whose $j$-th entry is the integral over $J$ of the $j$-th entry of the vector $\gamma'(u)$. 
	In particular, given a set of intervals $J_1,\ldots, J_n$ and a measurable function $\Xi$ supported on $\cup_{j=1}^n J_j$ with the property that $1 \leq |\Xi(t)| \leq n$ for all $t\in \cup_{j=1}^n J_j$, we define $A$ to be the $n \times n$ matrix whose $j$-th column is 
	\begin{equation}\label{A_col}
	\int_{J_j} \gamma'(u_j) |\Xi(u_j)| du_j.
	\end{equation}

	\begin{prop}\label{prop_4_det}
		Let $\gamma:[0,1]\to\R^n$ be a $C^n$ curve. Suppose $J_1,\dots,J_n$ are essentially disjoint closed intervals with $c(J_1)<\cdots<c(J_n)$.   Then for the $n\times n$ matrix $A$ defined above,
		\begin{itemize}
			\item[(a)] 
			\[ |\det(A)| \lesssim \Big(\prod_{j=1}^n |J_j|\Big) \Big(\prod_{1\le i<j\le n} (c(J_j)-c(J_i)) \Big). \]
			\item[(b)] Suppose furthermore that $\gamma$ is non-degenerate and that $\mathrm{diam}(\cup_{j=1}^n J_j)\le \delta_0$ where $\delta_0 = \delta_0(\ga,n)$ is as in Proposition~\ref{lem:detcond}. Then
			\[ |\det(A)| \gtrsim \Big(\prod_{j=1}^n |J_j|\Big) \Big(\prod_{1\le i<j\le n} (c(J_j)-c(J_i)) \Big). \]
			\item[(c)] Under the hypotheses of (b), there exists a constant $c_1 = c_1(\ga,n)$ such that the following holds. Let $R\ge 1$ and suppose that for some $1\le j_0\le n$, $|J_{j_0}|\ge c_1 R^{-1}$. Then   for every $v\in\R^n$ with $|v_{j_0}|\ge 1$ we have
			\[ |Av|\ge R^{-n}. \]
			In particular, $c_1$ depends only on $\ga,n$ and is independent of $J_1,\ldots, J_n$.
		\end{itemize}
	\end{prop}

	\subsection{Proof of Proposition \ref{lem:detcond}}

	The idea is to use the fact that the determinant is an alternating multilinear form and the mean value theorem.
	It will be convenient to first prove a general identity in this spirit, see \eqref{eqn:detindclaim} below. We use the following setup: for every integer $m\ge 1$ and real numbers $t_1<\cdots<t_m$ we define a non--negative measure $\sigma_{t_1,\dots,t_m}$ on $\R^m$ as follows. If $m=1$, then $\sigma_{t_1}$ is the Dirac measure at $t_1$, i.e.
	\[ \int_\R \varphi(u) d\sigma_{t_1}(u) = \varphi(t_1).\]
	If $m\ge 2$, then we define $\sigma_{t_1,\dots,t_m}$ recursively by
	\begin{equation}\label{eqn:defsigma} \int\limits_{\R^m} \varphi(u) d\sigma_{t_1,\dots,t_m}(u) = 
	\int\limits_{t_1}^{t_2} \cdots\int\limits_{t_{m-1}}^{t_m} \int\limits_{\R^{m-1}} \varphi(t_1, v) d\sigma_{s_2,\dots,s_m}(v) ds_m \cdots ds_2.
	\end{equation}
	Observe that $\sigma_{t_1,\dots,t_m}$ is supported on the compact set $\{t_1~\le~ u_1~\le~\cdots~\le~u_m~\le~ t_m\}$.
	We prove the following general statement about the measure $\sigma_{t_1,\ldots, t_m}$.
	\begin{lem}
		For every $m\ge 1$ and all real numbers $t_1<\dots<t_m$, the non--negative measure $\sigma_{t_1,\dots,t_m}$ defined above has the following properties:\\
		(i) For every alternating $m$--linear form $\Lambda:(\R^n)^m \to \R$ and every $C^{m-1}$ map $h:[a,b]\to \R^n$, for all $a\le t_1<\cdots<t_m\le b$ we have
		\begin{equation}\label{eqn:detindclaim}
		\Lambda ( h(t_1), \dots, h(t_m) ) = \int_{\R^m} \Lambda(h(u_1), h'(u_2),\dots, h^{(m-1)}(u_m)) d\sigma_{t_1,\dots,t_m}(u).
		\end{equation}
		(ii) The mass of $\sigma_{t_1,\dots,t_m}$ is given by
		\begin{equation}\label{eqn:sigmamass} \sigma_{t_1,\dots,t_m}(\R^m) = c_m\prod_{1\le i<j\le m} (t_j - t_i),
		\end{equation}
		where $c_m = \big(\prod_{j=1}^m (j-1)!\big)^{-1}$.
	\end{lem}
	
	\begin{proof}
		We first prove (i) by induction on $m$. For $m=1$ the claim follows immediately from the definitions. Let us assume the inductive hypothesis that (i) holds for dimension $m-1$, for all alternating $(m-1)$-linear functions, and every $C^{m-2}$ map. Now let us assume that $\Lambda$ is an $m$-linear function and $h$ is a $C^{m-1}$ map. Since $\Lambda$ is alternating we have
		\[ \Lambda(h(t_1),\dots,h(t_m)) = \Lambda(h(t_1), h(t_2)-h(t_1), \dots, h(t_m)- h(t_{m-1})). \]
		By the mean value theorem this equals
		\[\int_{t_1}^{t_2} \cdot\cdot \int_{t_{m-1}}^{t_m} \Lambda(h(t_1), h'(s_2), \dots, h'(s_m)) ds_m \cdots ds_2. \]
		Applying the inductive hypothesis to the $(m-1)$--linear form given by $\widetilde{\Lambda} ~=~\Lambda (h(t_1), \cdot)$ and the map $h'$ in place of $h$, we obtain that the previous expression is equal to
		\[ \int_{t_1}^{t_2} \cdot\cdot \int_{t_{m-1}}^{t_m} \int_{\R^{m-1}} \Lambda(h(t_1), h'(u_2), \dots, h^{(m-1)}(u_m) ) d\sigma_{s_2,\dots,s_m}(u_2,\dots,u_m) ds_m \cdots ds_2 \]
		which by the definition \eqref{eqn:defsigma} equals
		\[ \int_{\R^{m}} \Lambda(h(u_1), h'(u_2), \dots, h^{(m-1)}(u_m) ) d\sigma_{t_1,\dots,t_m}(u).\]
		To prove (ii) we apply (i) with $m=n$, $\Lambda=\det$ and $h=\gamma'$, where $\gamma$ is the normalized moment curve $\gamma(t)=(t,t^2/2,\dots,t^m/m)$. Then the left-hand side of \eqref{eqn:detindclaim} is equal to the Vandermonde determinant
		\[ \det(\gamma'(t_1),\dots,\gamma'(t_m)) = \prod_{1\le i<j\le m} (t_j - t_i), \]
		while the right-hand side can be explicitly computed in this case as 
		\[ \int_{\R^m} \det(\gamma'(u_1), \gamma''(u_2),\dots, \gamma^{(m)}(u_m)) d\sigma_{t_1,\dots,t_n}(u) = \big(\prod_{j=1}^m (j-1)!\big)\cdot \sigma_{t_1,\dots,t_{m}}(\R^m), \]
		which proves (ii).
	\end{proof}
	
	We now apply this lemma in the case $m=n, \Lambda=\det, h=\gamma'$ to prove Proposition \ref{lem:detcond}. Given $0<u_1 < \cdots < u_n < 1$,  the identity \eqref{eqn:detindclaim} shows that
\begin{equation}\label{det_ga}
\det( \ga'(u_1), \dots, \ga'(u_n) ) = \int_{\R^n} \det(\ga'(w_1), \ga''(w_2),\dots, \ga^{(n)}(w_n)) d\sigma_{u_1,\dots,u_n}(w),
\end{equation}
	with $d\sigma_{u_1,\dots,u_n}$ supported in $\{u_1 \leq w_1 \leq \cdots \leq w_n \leq u_n\}$. For (a), since the map
		\[(u_1,\dots,u_n)\mapsto \det(\gamma'(u_1),\dots,\gamma^{(n)}(u_n))\] 
		is continuous, the integrand is uniformly bounded from above by some $C=C(\ga,n)$ on the support of the measure, so that (\ref{det_ga}) is bounded above by $C d\sigma_{u_1,\dots,u_n}(\R^n)$, from which (a) follows via (\ref{eqn:sigmamass}) (upon redefining $C$ to be $c_mC$).
		 To prove (b), since $\ga$ is non-degenerate we may assume without loss of generality that \[\det(\gamma'(w),\gamma''(w),\dots,\gamma^{(n)}(w))>0\]
	holds for every $w\in [0,1]$. By uniform continuity there exists $\delta_0>0$ such that
	\begin{equation}\label{eqn:strongtorsioncondition}
	\det(\gamma'(w_1), \gamma''(w_2), \dots, \gamma^{(n)}(w_n)) \geq C' >0
	\end{equation}
	holds for all $w_1,\cdots,w_n\in [0,1]$ satisfying $\max_{j=1,\dots,n} |w_1-w_j|\le \delta_0$, which certainly holds for any $w$ in the support of $\sigma_{u_1,\dots,u_n}(w),$ under the assumption in (b) that $u_n- u_1 < \delta_0$. 
	Applying this in (\ref{det_ga}) yields the lower bound
	$ \geq C' \int_{\R^n}  \sigma_{u_1,\dots,u_n}(w)dw,$
	which implies (b).

	\subsection{Proof of Proposition \ref{prop_4_det}}

	First, we observe that 
	\[
	\det(A) = \int_{J_1} \dots \int_{J_n} |\Xi(u_1)| \dots |\Xi(u_n)| \det \left( \gamma'(u_1) \, \cdots \, \gamma'(u_n) \right) du_1 \dots du_n.
	\]
We first prove (b) explicitly. In this case, part (b) of Proposition \ref{lem:detcond} implies that  in the assumed support of the integral, $|\det \left( \gamma'(u_1) \, \dots \, \gamma'(u_n) \right)| $ always obeys the lower bound (\ref{eqn:detcond1}), which is nonzero except possibly on the boundary of the region of integration; this   allows us to assume without loss of generality that the determinant is non-negative for every $u_1 \in J_1,\ldots, u_n \in J_n$. 
 Since $|\Xi(u_j)| \geq 1$ for all $u_j \in J_j$ we may conclude from the identity above that
	\[
	\det(A) \geq \int_{J_1'} \dots \int_{J_n'}  \det \left( \gamma'(u_1) \, \cdots \, \gamma'(u_n) \right) du_1 \dots du_n,
	\]
	in which  $J_j'$ is the interval that has the same center as $J_j$, but only half the length of $J_j$, so in particular the $J_j'$ are pairwise disjoint. Now we invoke \eqref{eqn:detcond1} to estimate the integrand on the right hand side from below.
	Since for any $u_i \in J_i'$ and $u_j \in J_j'$ we have $u_j - u_i \geq (c(J_j) - c(J_i))/2$ whenever $j>i,$  and $|J_j'| = |J_j|/2$, the lower bound in (b) follows. 
	To prove (a), one may follow analogous reasoning, except we apply absolute values inside the integral, and apply the upper bound in \eqref{eqn:detcond0} in place of  the lower bound \eqref{eqn:detcond1}.
	
	Finally, for the proof of (c) we will   write $v=A^{-1}(Av)$, so that if we know that $v$ has a large entry in the $j_0$-th place yet we can show   that every entry in the $j_0$-th row of $A^{-1}$ is very small (under the assumption that $|J_{j_0}| \geq c_1R^{-1}$), then we must conclude that $|Av|$ cannot also be very small.
	To compute $A^{-1}$ we will
	 make use of Cramer's rule, $A^{-1} = (\det A)^{-1} \mathrm{Cf}(A)^T$, in which we recall that the $i$-th entry in the $j$-th column of the cofactor matrix $\mathrm{Cf}(A)$ is given by the determinant of the $(n-1)\times (n-1)$ matrix $B_{ij}$ obtained by removing the $i$-th row and the $j$-th column from the matrix $A$. Thus to compute the $j_0$-th row of $A^{-1}$ we compute $\det B_{ij_0}$ for each $1 \leq i \leq n$. We apply the upper bound in (a) (for dimension $n-1$) to conclude that 
	\[ |\det(B_{i,j_0})| \lesssim \Big(\prod_{j\not=j_0} |J_j|\Big) \Big(\prod_{\substack{1\le j'<j\le n,\\j'\not=j_0, j\not=j_0}} (c(J_j)-c(J_{j'})) \Big). \]
On the other hand, $|\det A|$ satisfies the lower bound given in part (b), so upon taking the ratio as in Cramer's law, we see that 
each entry of the $j_0$th row of $A^{-1}$ is bounded above by
	\[
	C''  |J_{j_0}|^{-1} \prod_{\substack{1 \leq j \leq n \\ j \ne j_0}} |c(J_j)-c(J_{j_0})|^{-1},
	\]
in which $C'' = C''(\ga,n)$ is dependent only on $\ga,n$. We may now choose $c_1$ large enough so that under the hypothesis that $|J_{j_0}| \ge c_1 R^{-1}$, and consequently $|c(J_j)-c(J_{j_0})| \ge (c_1/2) R^{-1}$ for every $j \neq j_0$,  every entry in the $j_0$th row of $A^{-1}$ is bounded from above by $\frac{1}{100n} R^n$ (say). Now to conclude the argument, suppose that $|Av|<R^{-n}$ for some $v$ with $|v_{j_0}| \geq 1$. Writing $v=A^{-1}(Av)$, this implies $|v_{j_0}|\le \frac1 {100}$, a contradiction. This proves (c), completing the proof of the proposition.
	
	\section{Proof of Proposition \ref{prop_ga_diag} on essentially diagonal solutions }\label{sec_diagonal}

	Our proof of Proposition \ref{prop_ga_diag} will critically use Proposition \ref{prop_4_det}; let the constants $c_1=c_1(\gamma,n)$ and $\delta_0 = \delta_0(\gamma,n)$ be as specified in that proposition, and set $c_0 = nc_1$.  	We assume that $[0,1]$ has been dissected into intervals of length $R^{-1}$ denoted by $\{ R^{-1}[\ell,\ell+1]\,:\,0\le\ell<R\}$, and that all  intervals  in the following discussion belong to this set. We consider a collection $\mathcal{I}$ of such intervals for which (\ref{assps}) holds.	We will show that if the points $t_1,\ldots, t_n$ belong to intervals $I_1,\ldots, I_n$ and the points $s_{1} , \ldots ,s_{n}$ belong to intervals $I_1',\ldots, I_n'$, there is a  quantitative, strictly positive lower bound for 
	\[ \ga(t_1) +\cdots + \ga(t_n) - \ga(s_{1}) - \cdots - \ga(s_{n}) \]
unless the tuple $(I_1,\ldots, I_n)$  is a permutation of $(I_1',\ldots, I_n')$. 
	
	Fix tuples $(I_1,\dots,I_n)$ and $(I_1',\dots,I_n')$, and  fix $t_i\in I_i$ and $s_i\in I_i'$. By the fundamental theorem of calculus,
	\[
	\sum_{i=1}^n (\gamma(t_i) - \gamma(s_i) )= \sum_{i=1}^n \int_{s_i}^{t_i} \gamma'(t) dt = \int_0^1 \gamma'(t) \Xi(t) dt,
	\]
	where we define
	\begin{equation}\label{Xi_dfn}
	 \Xi(t) = \sum_{i=1}^n \chi_{[s_i, t_i)}(t). 
	 \end{equation}
	Here $\chi_{[a,b)}(t)$ is defined to equal $+1$ if $a\le t< b$ and $-1$ if $b\le t< a$ (and zero otherwise); this convention is chosen so that $\chi_{[a,b)}$ is always a right continuous function (even if $a > b$). 
	For the moment, let us denote by $J_i$ the interval $[s_i,t_i)$ if $s_i < t_i$ and  the interval $[t_i, s_i)$ it $t_i < s_i$. 
	
	To motivate how we proceed, let us assume temporarily that we are in the very special case in which the intervals $J_i$ are all  disjoint. Then $|\Xi(t)| \in \{0,1\}$ and hence
	\begin{equation}\label{Xi_decomp}
	\sum_{i=1}^n ( \gamma(t_i) - \gamma(s_i)) = \sum_{i=1}^n \varepsilon_i \int_{J_i} \gamma'(u_i) |\Xi(u_i)| du_i
	= \sum_{i=1}^n \varepsilon_i \int_{J_i} \gamma'(u_i)  du_i
	\end{equation}
	in which $\varepsilon_i \in \{\pm 1\}$ is the sign of $\Xi$ on $J_i$. Using the notation of the matrix $A$ defined column by column in (\ref{A_col}), we see that the right-hand side of (\ref{Xi_decomp}) is $Av$ for the vector $v = (\ep_1,\ep_2,\ldots,\ep_n)$. Since under the hypotheses of the proposition, $(I_1,\ldots, I_n)$ is not a permutation of $(I_1',\ldots, I_n')$, there exists some $j_0$ such that $I_{j_0} \neq I_{j_0}'$, so that by the separation condition, $|s_i - t_i| \geq c_0 R^{-1} \geq c_1R^{-1}$. Thus the conditions of Proposition \ref{prop_4_det} (c) are met, and we can conclude that in (\ref{Xi_decomp}) that $|Av| \geq R^{-n}$, thus proving (\ref{eqn:gammaest}) in this special case.
	
	The essential insight in proving Proposition \ref{prop_ga_diag} in full generality is that even when the intervals with endpoints defined by $s_i,t_i$ overlap, the support of $\Xi$ can be decomposed into $n$ essentially disjoint intervals, upon each of which $\Xi$ is only positive or only negative; consequently, a version of (\ref{Xi_decomp}) will  again be true.
	\begin{prop}\label{lemma_decomp}
With the collection $\mathcal{I}$ and constants $c_0, \delta_0$ as described above, fix tuples of intervals $(I_1,\ldots,I_n)$ and $(I_1',\ldots,I_n')$, as well as points $s_i \in I_i$ and $t_i \in I_i'$, and define $\Xi(t)$ as in (\ref{Xi_dfn}).  The support of $\Xi(t)$ can be written as a disjoint union of intervals, such that upon the interior of each interval, $\Xi(t)$ is either only positive or only negative. Moreover: 
		\begin{itemize}
			\item[(i)] if $\ell_0$ is the minimal number of intervals in such a disjoint union, then $\ell_0 \leq n$;\\
			\item[(ii)] if we denote these intervals by  $\widetilde{J}_1, \dots, \widetilde{J}_{\ell_0}$, then there exists $1\le j_0\le  \ell_0$ such that $|\widetilde{J}_{j_0}|\ge c_0 R^{-1}$.
			\item[(iii)] Consequently,  we may construct $n$ essentially disjoint closed subintervals $J_1, \dots, J_n$ of $[0,1]$, with $c(J_1) <  \dots < c(J_n)$, so that  for some $1 \leq j_0 \leq n$, $J_{j_0}$ has length $\geq (c_0/n) R^{-1}  = c_1 R^{-1}$, and so that for each $1 \leq j \leq n$, $\Xi$ is either only positive or only negative in the interior of $J_j$, with $1 \leq |\Xi| \leq n$ on $J_j$.
		\end{itemize}
	\end{prop}
	Once we have obtained such  a decomposition of the support of the function $\Xi$,
	we can write a new version of (\ref{Xi_decomp}), that is
	$$
	\sum_{i=1}^n ( \gamma(t_i) - \gamma(s_i)) = \sum_{j=1}^n \varepsilon_j \int_{J_j} \gamma'(u_j) |\Xi(u_j)| du_j
	$$
	in which $\varepsilon_j \in \{\pm 1\}$ is the sign of $\Xi$ on $J_j$. 
	The final  step is to apply Proposition \ref{prop_4_det}: the right-hand side is the expression $Av$ for the vector $v = (\ep_1,\ep_2,\ldots,\ep_n)$. Since by Proposition \ref{lemma_decomp} (iii) we know that some $J_{j_0}$ has length at least $c_1R^{-1}$, we may conclude by Proposition \ref{prop_4_det} (c) that $|Av| \geq R^{-n}$, which verifies our desired inequality \eqref{eqn:gammaest}.
	It only remains to prove Proposition \ref{lemma_decomp}, which will occupy the remainder of this section.
	
	\subsection{Decomposition of the support of $\Xi$}
	
	\begin{proof}[Proof of Proposition \ref{lemma_decomp} (i)]
		 
		We note first of all that such a dissection of the support of $\Xi(t)$ into some finite number of intervals $\tilde{J}_1,\ldots , \tilde{J}_{\ell_0}$ exists,   because $\Xi(t)$ is a piecewise constant right continuous function. 
		We must only  show that when $\ell_0$ is chosen minimally, the decomposition can be made so that $\ell_0 \leq n$. 
		
		In this step, for convenience, we will sometimes write $s_{n+j}$ for $t_j$, if $1 \leq j \leq n$.
	Discontinuities of $\Xi(t)$ occur only at points in the set $\{ s_1,\ldots, s_{2n}\}$. By the minimality of $\ell_0$, $\Xi(t)$ must be discontinuous at the endpoints of every $\tilde{J}_j$, and thus the endpoints of each $\tilde{J}_j$ must be in the set $\{s_1, \dots, s_{2n}\}$. More precisely, if $\Xi(t)$ is positive on $\tilde{J}_j$, then the left endpoint of $\tilde{J}_j$ is in $\{s_1, \dots, s_n\}$, and the right endpoint of $\tilde{J}_j$ is in $\{t_1, \dots, t_n\} = \{s_{n+1}, \dots, s_{2n}\}$; we choose indices $l_j \in \{1, \dots, n\}$ and $r_j \in \{n+1, \dots, 2n\}$ such that the left endpoint and right endpoint of $\tilde{J}_j$ are $s_{l_j}$ and $s_{r_j}$ respectively. There may be more than one such choice of $l_j$ and $r_j$, and in that case we just make one choice and fix it once and for all. We call temporarily $L_+ \subset \{1, \dots, n\}$ the set of all $l_j$'s obtained from these intervals where $\Xi$ is positive, and $R_+ \subset \{n+1, \dots, 2n\}$ the set of all $r_j$'s obtained from these intervals where $\Xi$ is positive. To proceed further, if $\Xi(t)$ is negative on some $\tilde{J}_{j'}$, then the left endpoint of $\tilde{J}_{j'}$ is in $\{t_1, \dots, t_n\} = \{s_{n+1}, \dots, s_{2n}\}$, and the right endpoint of $\tilde{J}_{j'}$ is in $\{s_1, \dots, s_n\}$; we choose $l_{j'} \in \{n+1, \dots, 2n\} \setminus R_+$ and $r_{j'} \in \{1, \dots, n\} \setminus L_+$ such that the left and right endpoints of $\tilde{J}_{j'}$ are $s_{l_{j'}}$ and $s_{r_{j'}}$ respectively. This is possible, because if say the left endpoint of $\tilde{J}_{j'}$ is equal to $s_{r_j}$ for some $r_j \in R_+$, then the left endpoint of $\tilde{J}_{j'}$ is also the right endpoint of $\tilde{J}_j$ for some $\tilde{J}_j$ over which $\Xi$ is positive; in particular, there exists $p \in \{n+1, \dots, 2n\}$ with $p \ne r_j$ so that $s_p = s_{r_j}$, and we can simply pick $l_{j'} = p$.
	Similarly if the right endpoint of $\tilde{J}_{j'}$ is equal to $s_{l_j}$ for some $l_j \in L_+$, then the right endpoint of $\tilde{J}_{j'}$ is also the left endpoint of $\tilde{J}_j$ for some $\tilde{J}_j$ over which $\Xi$ is positive; in particular, there exists $q \in \{1, \dots, n\}$ with $q \ne l_j$ so that $s_q = s_{l_j}$, and we can simply pick $r_{j'} = q$. Altogether, one can check that $l_1, \dots, l_{\ell_0}, r_1, \dots, r_{\ell_0}$ is a list of distinct elements of $\{1, \dots, 2n\}$, so $2\ell_0 \leq 2n$, i.e. $\ell_0 \leq n$, proving the claim.
	\end{proof}

	\begin{proof}[Proof of Proposition \ref{lemma_decomp} (ii)]
	 
		The proof of Proposition \ref{lemma_decomp} (ii) relies on the following combinatorial fact, which we will prove at the end of this section.
		\begin{lem}\label{lem:comb}
			Let $(x_1,\dots,x_n)$ and $(y_1,\dots,y_n)$ be two lists of real numbers. Define $\mathbf{\chi}_{[x_i,y_i)}(t)$  to be $+1$ if $x_i\le t< y_i$ and $-1$ if $y_i\le t< x_i$ and $0$ otherwise. 
			Suppose that
			\begin{equation}\label{eqn:deftheta}
			\Theta(t) = \sum_{i=1}^n \mathbf{\chi}_{[x_i,y_i)}(t) = 0\quad\text{for all $ t\in\R$.}
			\end{equation}
			Then $(x_1,\dots,x_n)$ is a permutation of $(y_1,\dots,y_n)$.
		\end{lem}
		
		We assume this lemma for the moment, and verify part (ii) of Proposition \ref{lemma_decomp}.
		For every $i$, let $x_i=c(I_i)$ and $y_i=c(I_i')$ denote the centers of the intervals, and define the function $\Theta$ as in \eqref{eqn:deftheta}.	 Note that although for some values of $t$, $\Theta(t)$ may differ from $\Xi(t)$ as defined in (\ref{Xi_dfn}), we do have $\Theta(t) = \Xi(t)$ for $t\in  (\bigcup_{i=1}^n I_i\cup I_i')^c$; thus while the extra symmetry of $\Theta(t)$ aids us in establishing its properties, we may deduce useful consequences for $\Xi$ as well.
		
		 Since $(I_1,\dots,I_n)$ is not a permutation of $(I_1',\dots,I_n')$ we conclude that $(x_1,\dots,x_n)$ is not a permutation of $(y_1,\dots,y_n)$. By Lemma \ref{lem:comb} the function $\Theta$ does not vanish identically, and moreover we can pick a point $t_0\in (\bigcup_{i=1}^n I_i\cup I_i')^c$ such that $\Theta(t_0)\not=0$. (To see this, recall that $\Theta$ can have discontinuities only at $x_1,\dots,x_n,y_1,\dots,y_n$ and that distinct intervals in $\mathcal{I}$ are separated by at least $c_0R^{-1}$.)
	Furthermore, $\Theta$ is constant on each component of $(\cup_{i=1}^n I_i \cup I'_i)^c$, and since 
	  distinct intervals in $\mathcal{I}$ are separated by at least $c_0 R^{-1}$, the component, say $\widetilde{J}_{j_0}$, in which $t_0$ is contained must be at least of length $c_0 R^{-1}$. From this we deduce that $\Xi(t_0) \neq 0$, and $\Xi$ is also constant on $\widetilde{J}_{j_0}$, which suffices to prove (ii) of Proposition \ref{lemma_decomp}. 
	\end{proof}
	
	\begin{proof}[Proof of Proposition \ref{lemma_decomp} (iii)]
		If $\ell_0 = n$, then (iii) already has been verified. Otherwise, if $\ell_0 < n$, we choose some $j$ and split $\tilde{J}_j$ up into $n-\ell_0+1$ essentially disjoint closed intervals of positive length, obtaining exactly $n$ essentially disjoint closed subintervals $J_1, \dots, J_n$ of $[0,1]$, with  the properties specified in (iii). 
	\end{proof}

	\begin{proof}[Proof of Lemma \ref{lem:comb}]
		We may assume without loss of generality that for each $1 \leq i \leq n$, $x_i\not=y_i$, since removing such pairs from the lists does not change the value of $\Theta$ at any point, and the tuple $(x_1,\ldots, x_n)$ is a permutation of $(y_1,\ldots,y_n)$ if and only if the remaining values are a permutation, after the matching $x_i=y_i$ have been removed. 
		
		Let us write $ \{t_1<\cdots<t_m \}$ for the ordered set of distinct values taken on by any of $x_1,\ldots, x_n$ or $y_1,\ldots, y_n$.
		Denote by $\xi_k$ the number of times that $t_k$ appears in the list of $x_i$'s and by $\eta_k$ the number of times that $t_k$ appears in the list of $y_i$'s. Then it suffices to show that $\xi_k=\eta_k$ for all $1\le k\le m$. 
		We proceed by induction on $m$; we may assume that $m\ge 2$ (since  $m=1$ would require all $x_i$ and $y_i$ to be equal, a case we have ruled out).

		Given $m \geq 2$, we observe that
		\[ \Theta(t_{m-1}) = \# \{ i\,:\,x_i\le t_{m-1}<y_i\} - \#\{i\,:\,y_i\le t_{m-1}<x_i \} = \eta_m - \xi_m \]
		because $x_i\le t_{m-1}<y_i$ if and only if $y_i=t_m$ (since $x_i\not=y_i$) and $y_i\le t_{m-1}< x_i$ if and only if $x_i=t_m$. Assuming $\Theta$ is identically zero, this shows $\xi_m=\eta_m$. 
		Of course, $\xi_1 + \cdots + \xi_m = \eta_1 + \cdots + \eta_m=n$. In the case $m=2$, these two relations suffice to show that $\xi_1 = \eta_1$ and $\xi_2 = \eta_2$.  Now we assume the induction hypothesis that the claim is true if the set of distinct values has at most $m-1$ elements. Then supposing the set of distinct values is $ \{t_1<\cdots<t_m \}$, define new lists $\widetilde{x}$, $\widetilde{y}$ as follows:
		\[ \widetilde{x}_i = \left\{\begin{array}{ll}x_i & \text{if}\;x_i< t_m\\
		t_{m-1} & \text{if}\;x_i=t_m\end{array} \right.,\quad
		\widetilde{y}_i = \left\{\begin{array}{ll}y_i & \text{if}\;y_i< t_m\\
		t_{m-1} & \text{if}\;y_i=t_m\end{array} \right.. \]
		Then the distinct values taken on by elements in $(\widetilde{x}_1, \dots,\widetilde{x}_n)$ or  $(\widetilde{y}_1, \dots, \widetilde{y}_n )$ give precisely the ordered set  $\{t_1<\dots<t_{m-1}\}$.
		We also claim that $\widetilde{\Theta}(t) = \sum_{i=1}^n \chi_{[\widetilde{x}_i, \widetilde{y}_i)}(t) = 0$ for every $t\in\R$. Indeed, $\widetilde{\Theta}(t)=0$ if $t<t_1$ or $t\ge t_{m-1}$. On the other hand, if $t_1\le t< t_{m-1}$, then
		\[ \chi_{[\widetilde{x}_i, \widetilde{y}_i)}(t) = \chi_{[x_i,y_i)}(t). \]
		Therefore, $\widetilde{\Theta}(t) = \Theta(t) = 0$. Applying the inductive hypothesis, we obtain $\xi_k=\eta_k$ for all $k=1,\dots,m-2$ and also $\xi_{m-1}+\xi_m = \eta_{m-1} + \eta_m$, which implies $\xi_{m-1}=\eta_{m-1}$ because we already showed $\xi_m=\eta_m$.
	\end{proof}

	\section{Proof of Theorem \ref{thm:main}: The square function estimate}

We first sparsify our collection of intervals. Given a non-degenerate curve $\ga$, we fix a sufficiently large constant $c_0 =c_0(\ga,n)\ge 10$  and a constant $\delta_0  = \delta_0(\ga,n)$ as in Proposition  \ref{prop_ga_diag}. From now on we let $\mathcal{I}$ denote a collection of intervals, chosen from our initial collection of intervals $\{ R^{-1}[\ell,\ell+1]\,:\,0\le\ell<R\}$, such that
	\[ \mathrm{dist}(I,I')\ge c_0 R^{-1} \quad\text{for}\;I\not=I'\in\mathcal{I}, \quad\text{and}\quad \mathrm{diam}\big(\bigcup_{I\in\mathcal{I}} I\big)\le \delta_0. \]
We can cover $[0,1]$ by taking at most $(c_0+1)\delta_0^{-1} $ such collections $\mathcal{I}$. We will prove for each $1 \leq m \leq n$ and for each such collection that
\[	\big\|\sum_{I\in\mathcal{I}} E_I f\big\|_{L^{2m}(w_B)} \lesssim \big\|\big(\sum_{I\in\mathcal{I}} |E_I f|^2\big)^{1/2}\big\|_{L^{2m}(w_B)};
\]	
 summing over such collections contributes only to the constant $C$ on the right-hand side of (\ref{eqn:main}).
	
	By a  standard reduction regarding weighted norms \cite[Lemma 4.1]{BD17}, it now suffices to show that
	\begin{equation}\label{eqn:penult}
	\big\|\sum_{I\in\mathcal{I}} E_I f\big\|_{L^{2m}(\mathbf1_B)} \lesssim \big\|\big(\sum_{I\in\mathcal{I}} |E_I f|^2\big)^{1/2}\big\|_{L^{2m}(w_B)},
	\end{equation}
	where $\mathbf{1}_B$ denotes the characteristic function of $B$. Without loss of generality we may assume that the ball $B$ is centered at the origin (see e.g. \cite[p. 58]{Pie19}).
	Let $\varphi$ be a non-negative Schwartz function on $\mathbb{R}^n$ so that $\varphi \ge 1$ on the unit ball centered at $0$ and $\widehat{\varphi}$ is supported on the unit ball centered at $0$. (To construct such a function, let $\psi$ be such that $\psi \in C^\infty_c(B(0,1/4))$ and $\int \psi (\xi)d\xi  > 1$. Then define $\varphi$ by $\widehat{\varphi} =\psi* \overline{\psi(-\cdot)}$ so that $\varphi = |\widehat{\psi}|^2$, and in particular $\varphi(0) = |\widehat{\psi}(0)|^2 >1$; this continues to hold in some small neighborhood of the origin, and by redefining $\varphi$ appropriately after a fixed rescaling, we can ensure $\varphi(x) \geq 1$ on the unit ball.) Denote $\varphi_R(x)=\varphi(R^{-n} x)$. We will prove that
	\begin{equation}\label{eqn:goal}
	\big\|\sum_{I\in\mathcal{I}} E_I f\big\|^{2m}_{L^{2m}(\varphi_R)} =\big\|\big(\sum_{I\in\mathcal{I}} |E_I f|^2\big)^{1/2}\big\|^{2m}_{L^{2m}(\varphi_R)},
	\end{equation}
	which suffices to verify \eqref{eqn:penult}.
	
	The left-hand side of \eqref{eqn:goal}  is equal to
	\begin{equation}
	\sum_{I_1, \dots, I_m} \sum_{I_1', \dots, I_m'} \int_{\mathbb{R}^n} \varphi(R^{-n} x) E_{I_1} f(x) \cdots E_{I_m} f(x) \overline{E_{I_1'} f(x) \cdots E_{I_m'} f(x)} dx. \label{eq:expansion}
	\end{equation}
	For fixed collections of intervals $I_1,\dots,I_m,I'_1,\dots,I'_m$, expanding the extension operators shows that the contribution to the integral is equal to
	\begin{align}\label{blur_demo}
	\int_{I_1 \times \dots \times I_m} \int_{I_1' \times \dots \times I_m'} R^{n^2}\widehat{\varphi}\big(R^n \sum_{i=1}^m( \gamma (t_i)-\gamma(s_i)) \big) f(t_1) \dots f(t_m) \overline{f(s_1) \dots f(s_m)} dt_1 \cdots dt_m ds_1 \cdots ds_m.
	\end{align}
	Suppose that $(I_1,\ldots, I_m)$ is not a permutation of $(I_1',\ldots, I_m')$.
	In order to enlarge these to two $n$-tuples of intervals, choose an arbitrary $J\in \mathcal{I}$ and set $I_i=I'_i=J$ for all $m<i\le n$. 
	Then we apply Proposition \ref{prop_ga_diag} to conclude that
	 \[ \big|\sum_{i=1}^n (\gamma(t_i)-\gamma(s_i))\big|\ge R^{-n} \]
	holds for all $(t_1,\dots,t_n,s_1,\dots,s_n)\in I_1\times \cdots\times I_n\times I_1'\times\cdots\times I_n'$.
	In particular, upon setting $s_i =c(I_i), t_i = c(I_i')$ for the auxiliary intervals with $m< i \le n$, we deduce that 
		\[ \big|\sum_{i=1}^m (\gamma(t_i)-\gamma(s_i))\big| = \big|\sum_{i=1}^n (\gamma(t_i)-\gamma(s_i))\big|\ge R^{-n} \]
		holds for all $(t_1,\dots,t_m,s_1,\dots,s_m)\in I_1\times \cdots\times I_m\times I_1'\times\cdots\times I_m'$.
	This implies
	\begin{equation}\label{phi_vanish}
	\widehat{\varphi}\big(R^n \sum_{i=1}^m( \gamma (t_i)-\gamma(s_i)) \big)=0.
	\end{equation}
	Thus the only terms remaining in (\ref{eq:expansion}) are precisely  
	\[
	\int_{\mathbb{R}^n} \varphi(R^{-n} x) \big(\sum_{I\in\mathcal{I}} |E_{I} f(x)|^2\big)^m dx = \big\|\big(\sum_{I\in\mathcal{I}} |E_I f|^2\big)^{1/2}\big\|^{2m}_{L^{2m}(\varphi_R)},
	\]
	which completes the proof of \eqref{eqn:goal}, and hence of the theorem.


\begin{thebibliography}{ABCD99}
		\bibitem[ACK79]{ACK79} G.I. Arkhipov, V.N. Chubarikov, A.A. Karatsuba. \emph{Exponent of convergence of the singular integral in the Tarry problem.} (Russian) Dokl. Akad. Nauk SSSR 248 (1979), no. 2, 268–272. 
		
		\bibitem[ACK87]{ACK87} G.I. Arkhipov, V.N. Chubarikov, A.A. Karatsuba.
		Trigonometric sums in number theory and analysis. Translated from the 1987 Russian original. De Gruyter Expositions in Mathematics, 39. Berlin, 2004.
		
		\bibitem[BOS09]{BOS09} J.-G. Bak, D.M. Oberlin, A. Seeger. \emph{Restriction of Fourier transforms to curves and related oscillatory integrals.} Amer. J. Math. 131 (2009), no. 2, 277–311. 
		
		\bibitem[BDG16]{BDG16} J. Bourgain, C. Demeter, L. Guth. \emph{Proof of the main conjecture in Vinogradov's mean value theorem for degrees higher than three.} Ann. of Math. (2) 184 (2016), no. 2, 633–682. 
		
		\bibitem[BD17]{BD17} J. Bourgain, C. Demeter. \emph{A study guide for the $\ell^2$ decoupling theorem.} Chin. Ann. Math. Ser. B 38 (2017), no. 1, 173–200.
		
		\bibitem[DLW09]{DLW09} S. Dendrinos, N. Laghi, J. Wright. \emph{Universal $L^p$ improving for averages along polynomial curves in low dimensions.} J. Functional Analysis 257 (2009) 1355--1378.
		
		\bibitem[DW10]{DW10} S. Dendrinos, J. Wright. \emph{Fourier restriction to polynomial curves I: a geometric inequality.} \emph{American Journal of Mathematics} 132 no. 4 (2010) 1031--1076.
		
		\bibitem[Dru85]{Dru85} S. W. Drury. \emph{Restrictions of Fourier transforms to curves.} Ann. Inst. Fourier (Grenoble) 35 (1985), no. 1, 117--123. 
		
		
		\bibitem[Fef73]{Fef73} C. Fefferman. \emph{A note on spherical summation multipliers.} Israel J. Math. 15 (1973), no. 1, 44--52.
		
		\bibitem[Glo44]{Glo44} A. Gloden, \emph{Mehrgradige Gleichungen,} Groningen, P. Nordhoff, 1944.
		
		\bibitem[HW08]{HW08} G. H. Hardy and E. M. Wright. Introduction to the theory of numbers, 6th edition, revised by D. R. Heath-Brown and J. H. Silverman. Oxford University Press, Oxford. 2008.
		\bibitem[Mor32]{Mor32} L.J. Mordell, \emph{On a sum analogous to a Gauss's sum,} The Quarterly Journal of Mathematics, Volume  3, Issue 1 (1932)  161--167
		
		\bibitem[Pie19]{Pie19} L. B. Pierce. \emph{The Vinogradov Mean Value Theorem
			[after Wooley, and Bourgain, Demeter, Guth]}. S\'eminaire Bourbaki (volume 69, 2016/2017, expos\'e 1134), 
	\emph{Ast\'erisque}, (2019) volume 407.
		
		\bibitem[SW71]{SW71} E. M. Stein, G. Weiss. Introduction to Fourier analysis on Euclidean spaces. Princeton Mathematical Series, No. 32. Princeton University Press, Princeton, N.J., 1971. x+297 pp.
		
		\bibitem[VW97]{VW97} R.C. Vaughan and T.D. Wooley, \emph{A special case of Vinogradov's mean value theorem.} Acta Arithmetica LXXIX.3 (1997) 193--204.
		
		\bibitem[Woo16]{Woo16} T.D. Wooley, \emph{The cubic case of the main conjecture in Vinogradov's mean value theorem.} Adv. Math. 294 (2016) 532--561.
	
		\bibitem[Woo17]{Woo17} T.D. Wooley, \emph{Nested efficient congruencing and relatives of Vinogradov's mean value theorem.} arXiv:1708.01220.
			
			
	\end{thebibliography}
\end{document}